\documentclass[journal]{IEEEtran}
\usepackage{cite}

      \usepackage{graphicx}
\usepackage{amsmath,amssymb,amsthm}
\usepackage{url}

\usepackage{adjustbox}
\usepackage{cuted}
\newcommand{\sr}{soft feedback}

\newcommand{\SR}{Soft Feedback}
\def\bx{\mathbf{x}}
\def\bv{\mathbf{v}}
\def\bfomega{\boldsymbol{\omega}}
\def\norm#1{\Vert #1 \Vert}
\def\ones{\mathbf{1}}

\newtheorem{assumption}{Assumption}
\newtheorem{theorem}{Theorem}

\begin{document}
%
\title{\huge Social influence makes self-interested crowds smarter:\\an optimal control perspective}
%
%
%

\author{Yu Luo$^1$,
        Garud Iyengar$^2$, 
        Venkat Venkatasubramanian$^1$$^\ast$
\thanks{
$^1$Department of Chemical Engineering, Columbia University
}%
\thanks{
$^2$Department of Industrial Engineering and Operations Research, Columbia University
}%
\thanks{$^\ast$To whom correspondence should be addressed. 
Email: venkat@columbia.edu}
}

\maketitle

\begin{abstract}

It is very 
common to observe crowds of individuals 
solving similar problems 
with similar information in a largely independent manner. 
We argue 
here
that crowds can become ``smarter,'' i.e.,~more efficient and robust, by  partially
following the average opinion. This observation runs counter
to the widely accepted claim that the wisdom of crowds deteriorates with
social influence. 
The key difference is that individuals
are self-interested and hence will 
reject feedbacks that do not improve
their performance. We propose a control-theoretic methodology to compute the degree
of social influence, i.e.,~the level to which one accepts the population
feedback, 
that optimizes performance.
We conducted an 
experiment 
with human subjects
($N=194$), where 
the
participants 
were first asked
to 
solve
an optimization problem
independently, i.e.,~under no social influence. 
Our theoretical methodology estimates a $30\%$ degree of social influence to be
optimal,  
resulting in a $29\%$ improvement in the crowd's performance. 
We then let the same cohort solve a new problem and have access to the
average opinion. Surprisingly, we find the average degree of social
influence in the cohort to be  $32\%$  with a $29\%$ improvement in performance: In
other words, the 
crowd self-organized into a near-optimal setting. 
We believe this new paradigm for making crowds ``smarter'' has the
potential for making a significant impact on a diverse set of fields
including population health to government planning.
We include a case study to show how a crowd of states can collectively
learn the level of taxation and expenditure 
that optimizes economic growth. 

\end{abstract}


%
\IEEEpeerreviewmaketitle


\section{Introduction}\label{section1}


\IEEEPARstart{O}ften, 
large \emph{crowds} of decision makers are attempting to solve 
the same problem with similar information in a largely independent manner. For the common man, these problems could be as simple as choosing the most appropriate product 
or improving personal fitness. 
For a crowd 
of local governments or nations, the problem could be 
optimal taxation to  promote economic growth. 
The process of 
identifying the appropriate decision 
involves an expensive trial and error process to explore the entire space. 
Minimizing this search cost by  
coordinating and 
improving this {\em collective learning} process, by making crowds ``smarter,''
has immense societal value.

Optimization typically involves balancing trade-offs.
Consider the problem of optimal taxation.
Under-taxation results in insufficient funds towards public services and government functioning, whereas over-taxation 
drives businesses to places where taxes are lower, leading once again to a deficit for the state.
Local governments face similar dilemma 
when setting 
expenditure 
 to balance between under- and over-spending.
To illustrate the learning process 
to set the optimal taxation and expenditure, in Fig.~\ref{fig:TaxExp} we plot the license tax as a fraction of the total state revenue from 1946 to 2014, and the secondary education expenditure as a fraction of the total state spending from 1977 to 2013. 
The 
trajectories appear to have converged in the last decade. A large majority of the states have 
converged to the same decision. 
The main question we address in this paper is whether one can accelerate convergence by making the crowd of fifty states ``smarter.'' 
Even a small improvement in the convergence rate, magnified by the scale of the problem, could potentially save the nation billions of dollars while improving the overall welfare.

\begin{figure*}
	\centering
	\includegraphics[width=0.45\textwidth]{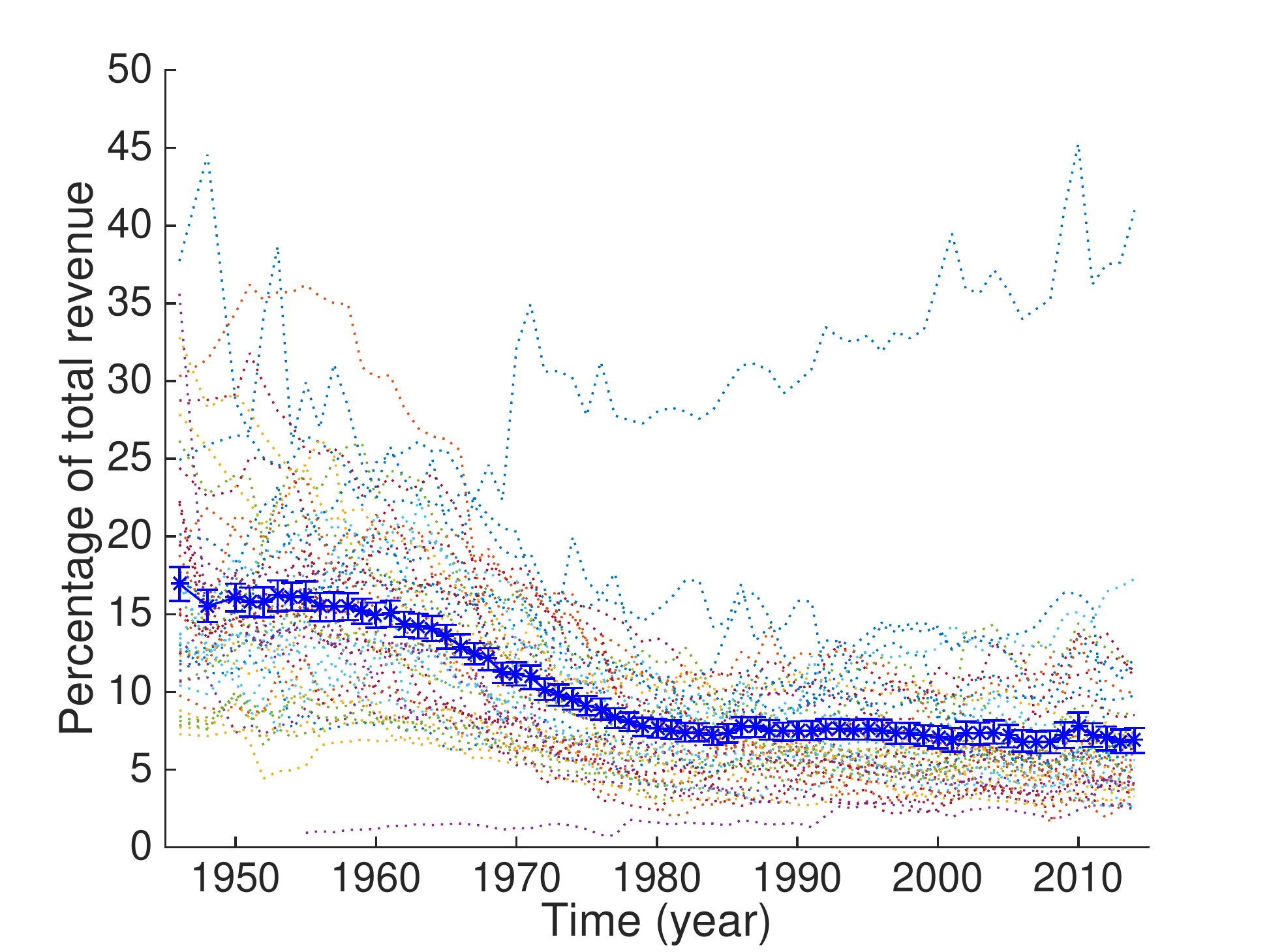}
	\includegraphics[width=0.45\textwidth]{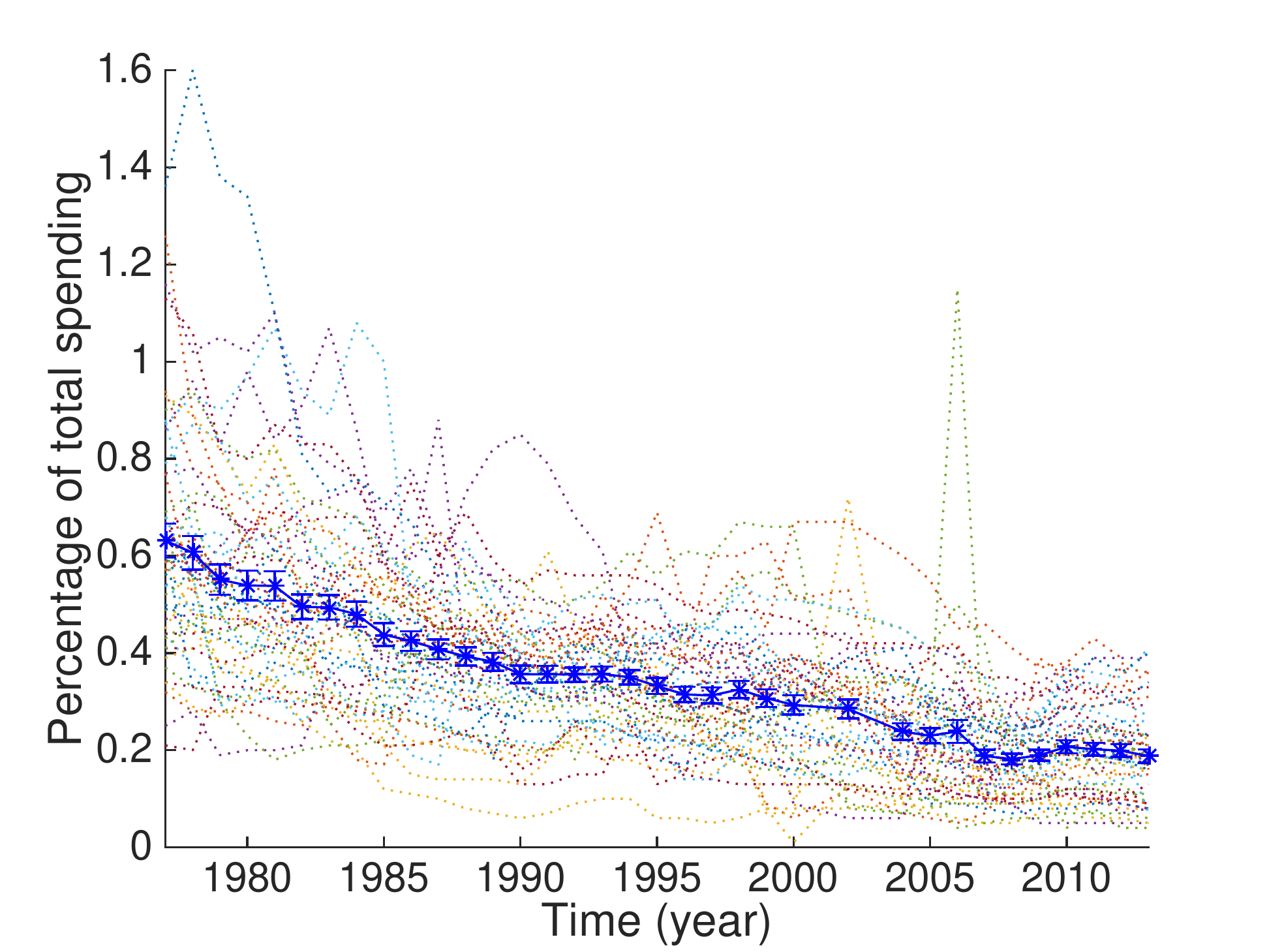}	
	\caption{Left: state tax percentage of total revenue (total license taxes, from 1946 to 2014)~\cite{bureau2015historical}. Right: state expenditure percentage of total spending (secondary education direct administrative expenditures, from 1977 to 2013)~\cite{center2015expenditure}. 
Each colored dashed line indicates the time series for one of the 50 states (and District of Columbia). The blue dotted line indicates the arithmetic mean. Error bars reflect the standard errors of the mean.}
	\label{fig:TaxExp}
\end{figure*}

Using a coordinated crowd or swarm to solve complex problems is well studied in the literature. 
Particle swarm optimization (PSO)~\cite{kennedy2010particle} 
is a widely adopted global optimization technique 
that uses a crowd of simple solvers to explore the fitness landscape of a problem. 
This swarm of PSO solvers mimics the swarming behavior observed in nature, e.g.,~among bees, ants, 
and birds. 
Each PSO solver revises its search direction based on its past performance and 
the position of the solver that observes the highest fitness.
The PSO technique is very effective in solving deterministic problems that have multiple local extrema. However, PSO or any other parallel computing methodology cannot help us in improving the rate for learning in the optimal taxation and expenditure setting. The critical difference is that in the PSO setting each solver observes the same function; however, the reward or fitness of an individual in a crowd is typically subjective, private, very noisy, and often, not even numerically expressible. On the other hand, the \emph{inputs} to the fitness function are numerically well defined. We exploit this feature to develop a learning algorithm.

Wisdom of crowds describes the phenomenon~---~first introduced as {\it vox populi} in 1907 by Francis Galton~\cite{galton1907vox}, then rediscovered and popularized by James Surowiecki a century later~\cite{surowiecki2005wisdom}~---~that the average opinion of a crowd is remarkably close to the otherwise unknown truth
although the opinions of individuals in the crowd are very erroneous. 
This phenomenon partially justifies the efficiency of polling and prediction markets, where a surveyor can gather an accurate estimate of an unknown variable by averaging over multiple independent and informed guesses. 
Explanations~\cite{de2014essai,bergman1964analysis,simons2004many} for 
the success of the wisdom of crowds 
assume that 
individuals' estimates are unbiased and \emph{independently} distributed~\cite{surowiecki2005wisdom,kittur2008harnessing,goldstone2009collective,lorenz2011social,quinn2011human,alvarez2011representing}. 
Social influence 
renders the wisdom of crowds ineffective
~\cite{goldstone2009collective,lorenz2011social,sumpter2009quorum}, and in order to guarantee accuracy, interactions among the respondents should be discouraged. Since individuals make decisions solely based on their prior knowledge and expertise, some even suggest {\it vox expertorum}, instead of {\it vox populi}, to be a more suitable name~\cite{goldstone2009collective,galton1907ballot,conradt2005consensus}. 

Increasingly, today individuals are getting all their information from 
highly inter-connected online social networks; 
thus, 
truly independent opinions are becoming rare. 
The existing literature suggests that {\it vox populi} should 
not be effective.
And yet,
online networks with very high degree of social interaction 
appear to be able to harness 
information effectively to benefit the individuals. 
We
are relying on polling evermore, for selecting movies, restaurants, books, shows, etc. The polls appear to be working in identifying good options, even though the votes are highly correlated. 
The crowd benefits from these interactions by converging to optimum faster. Social influence here improves, rather than undermines, the collective learning process. 
How does one reconcile with the previous results on the degradation of the impact of {\it vox populi} in the presence of social influence? {\em Is there an optimal degree of social influence for a learning crowd?} This is the question we address in this study. 

In this work, we use control theory 
to show that \emph{self-interested} decision makers can benefit by {\em partially} following the wisdom of crowds. 
Too little social influence 
prevents individuals to harness the wisdom of crowds effect; however, too high a social influence 
has an adverse effect on the accuracy of the wisdom of crowds.
The optimal degree of social influence balances these two effects. 
%

We designed a human subject experiment called the 
``Fitness Game'' that mimics the real-world situation where
individuals alter their diets to improve health. 
By analyzing experiment results, we identify the individual learning dynamics, determine the average degree of social influence when subjects 
partially follow the wisdom of crowds feedback,
and calculate the optimal degree of social influence that could have maximally improved the crowd's performance.

\section{Experiment Design}
We conducted an online experiment on Amazon Mechanical Turk with human subjects. 
There were three sets of experiments: B, N, and S. We focus our analysis on set B ($N=194$) only but present the final results for all three sets. 
Each set consisted of five replications of the experiment with its unique 
conditions. 

The participants (or players of the ``Fitness Game'') were asked to 
 estimate the ``diet level'' that maximizes the
``fitness'' of a virtual character. The true relationship between the diet
level and fitness was a given deterministic and concave function (i.e.,~there exists a unique diet level that maximizes the fitness); however, the players
received a noisy value of the fitness associated with the guessed diet
level. 
This noise, in reality, could be from other external factors such as environment and mood.
The players were allowed multiple guesses, and were rewarded
instantly based on the character's fitness level. 
The players also received monetary rewards based on their relative performances. 

We conducted five replications of the ``Fitness Game'' for each experiment set. In replication
$p\in \{1, \ldots, 5\}$, the $n_p$ participants first entered a session
where they played the game in an \emph{open loop} for 240 seconds (four
minutes). In this session, each participant entered a series of guesses to
best  predict the unknown optimal diet level $\theta^\ast \in
[2000,2500]$~kcal. When a player entered a guess for the optimal
$\theta^\ast$, the interface would refresh and the player would see the virtual character's  fitness level (maximum $100$\%) for the
guessed value. The player could then enter a new value until this session
ended.
The term open loop indicates that individual decisions did not interact with each other; thus, the {\it vox populi} feedback was absent. 

Subsequently, the same cohort entered the treatment session where they played the same game
with a population feedback. The game was reset and a new optimal
diet level $\theta^\ast$ was chosen. In this session, in addition to the
fitness level corresponding to their own guess $z$, players also
received a feedback saying ``We recommend $\frac{1}{n_p}\sum_{i=1}^{n_p}
{z}_{i}$ kcal,'' where ${z}_{i}$ denotes the most recent guess of
the $i$-th player. This feedback only updated when the
  players took actions. The players had the option of using the feedback
in any manner they desired. See Appendix~\ref{app_game} for detailed
descriptions of the ``Fitness Game'' interface.  

In this treatment
group, we revealed the population average of the diet level to each player.
Thus, the choices of the players were 
\emph{not} independent. 
However, we allowed the
players the freedom to accept, reject, or partially accept such
a population feedback,
i.e.,~set the diet level to be a combination of their
individual guesses and the 
feedback. 
We call this ``\sr{}'' in the sense that learners are 
allowed to choose the degree to which they adopt the feedback.

In a previous work~\cite{luo2016soft}, we  had introduced the possibility of partial acceptance of population recommendation in the  context of regulating emerging
industries. In the regulatory context, we termed
this as ``soft'' regulation in contrast to the conventional ``hard''
regulation where the regulated entities face fines and other punitive
consequences for non-compliance. 
In our current setting, the individuals are allowed to partially accept the population feedback. This contrasts feedback in control theory, which is hard in the sense that it has to be followed.
We showed that soft regulation
is appropriate and efficient (and  desirable) when the  observed
outcomes are very noisy, individual decision makers are rational
utility-maximizing agents, and the agents are exploiting abundant
resources, and therefore, not competing. 
Medical research and health
optimization using large-scale social interactions, for example via
Apple's ResearchKit and CareKit~\cite{apple2016researchkit}, are
examples of systems that satisfy these three conditions. The ``Fitness
Game'' is meant to mimic these conditions.  


Upon completion, participants 
received monetary rewards based on their relative game scores within the same cohort. We hoped to incentivize the participants in this way so that they would make rational decisions and actively optimize their virtual character's fitness, instead of making random guesses to get the participation rewards. 

\section{A Multi-Agent Control Model}
We propose the following state-space control model to describe the collective dynamics of an $n$-agent crowd in the open loop setting:
\begin{equation}
x_i(t+1)=g_i\big(x_i(t)\big)+\omega_i(t). \label{eq:open}
\end{equation}
In 
the \sr{} setting, we have
\begin{align}
x_i(t+1)&=(1-\beta_i)\Big(g_i\big(x_i(t)\big) + {\omega}_i(t)\Big)
+\beta_iu(t).
\label{eq:ScalarDynamics} 
\end{align}
where $x_i(t)$ is the state variable of the $i$-th agent $(i=1,\ldots,n)$ at time $t$; $g_i(\cdot)$ is the learning function; $\omega_i(t)$ is a zero-mean random variable; $u(t)$ is the 
soft feedback,
and $\beta_i$ is the {\em degree} of social influence.\\ 

\subsection{State $x_i(t)$ of the $i$-th Agent}
The state variable $x_i(t)=z_i(t)-\theta^\ast$ is the {\em decision error}, i.e.,~difference between 
the individual decision $z_i(t)$ 
and the optimal decision $\theta^\ast$. 
$x_i^*=0$ indicates the optimal state (or the solution). 

\subsection{Learning Function $g_i(\cdot)$ and Noise $\omega_i(t)$}
The learning function $g_i$ of
the $i$-th player 
encodes the process where the agent makes a decision,
observes the corresponding utility, and then updates the
state. 
We assume that the 
optimal state $x_i^\ast=0$ is an attracting and unique fixed point of $g_i(\cdot)$, i.e.,~$g_i(0)=0$. Thus, 
regardless of the optimization technique or the initial decision, an agent can always reach the optimum. We further assume that $g_i(x)$ is differential and
$|g'_i(x)|<1$, i.e.,~$g_i(x)$ is a
contraction~\cite{browder1965fixed}. The closer $|g_i'(x)|$ is to $1$, the
slower $g_i(x)$ converges. 
From mean value theorem, we can also establish that $g_i(x)/x=g_i'(\delta
x)$, where $0\leq\delta\leq 1$, is strictly less than 1. 
We define learning gain, denoted by $\tilde{g}_i'\equiv g_i(x)/x$, as the amplification of decision error.

$\omega_i(t)$ is a zero-mean 
random variable with variance $\sigma_\omega^2$ sampled at $t$. 
It represents the impact of the error in function evaluation on the decision. Such error can be a result of noise in measurement or external disturbance.

\subsection{Soft Feedback $u(t)$}
We denote the soft
feedback
as the population average:
\begin{equation}
u(t) = \frac{1}{n}\sum_{j=1}^nx_j(t)
.\end{equation}
Again, unlike feedback in control theory, soft feedback does not have to be followed.
The value of $u$ represents the error of the wisdom of crowds. 

\subsection{Degree of Social Influence $\beta_i$}
The degree of social influence denotes the weight the $i$-th player places on the
 feedback
while learning. $\beta_i=0$ 
reduces the \sr{} setting in \eqref{eq:ScalarDynamics} to the open loop setting in \eqref{eq:open}.


\subsection{Convergence of the \SR{} Mechanism}
We have previously established the following properties of 
\sr{}~\cite{luo2016soft}: 
Social influence does
not destabilize the system, nor does it alter the convergence provided $0\leq\beta_i < 100\%$. 
We can write the noiseless \sr{}
dynamics by replacing $g_i(\cdot)$ with the learning gain $\tilde{g}_i'(t)$ computed at $t$:
\begin{equation}
x_i(t+1)=(1-\beta_i)\tilde{g}_i'(t)x_i(t) + \beta_i u(t).\label{eq:simpledynamics}
\end{equation}
In vector form, we have 
\begin{equation}
\mathbf{x}(t+1) = (I - B)G'(t)\mathbf{x}(t)+BS\mathbf{x}(t),
\end{equation}
where $\mathbf{x}(t)=[x_1(t),\ldots,x_n(t)]^\top$,
$B=\operatorname{diag}(\beta_1,\ldots,\beta_n)$,
$G'(t)=\operatorname{diag}\big(\tilde{g}_1'(t),\ldots,\tilde{g}_n'(t)\big)$, and
$S=\frac{1}{n}\mathbf{1}\mathbf{1}^\top$. 
It is
easy to identify that the largest eigenvalue of the matrix
$(I-B)G'(t)+BS$ is {\em always} strictly less than 1 if $0\leq\beta_i<100\%$. This
implies the \sr{} dynamics also converges to the solution 
and is robust against bounded noise. 
See Appendix~\ref{app_math} for detailed proofs. 

\subsection{Efficiency of the \SR{} Mechanism}
We define the following optimal control problem for 
computing the optimal degree of social influence $\beta_i$ that minimizes the cost function $V$:
\begin{eqnarray}
  \min_{B}\quad
            V\big(B;\mathbf{x}(0),\boldsymbol\omega(t),T\big)=\mathbb{E}\left[
            \sum_{t=0}^{T-1}\frac{1}{n}\mathbf{x}(t)^\top\mathbf{x}(t)\right],  \label{eq:control1}\\ 
  \operatorname{s.t.}\quad 
  \mathbf{x}(t+1) = (I -
                       B)\big(G'(t)\mathbf{x}(t)+\boldsymbol\omega(t)\big)
                       +BS\mathbf{x}(t),
                        \label{eq:control2} 
\end{eqnarray}
where $\boldsymbol\omega(t) = [\omega_1(t),\ldots,\omega_n(t)]^\top$ denotes a series of noise vectors. The above
optimal control problem 
minimizes the 
cumulative expected mean squared errors (MSE) over a finite time
horizon $T$.

The time evolution of the total MSE in \eqref{eq:control1} depends on two
factors: the rate of convergence controlled by
$\big((I-B)G'(t)+BS\big)\mathbf{x}(t)$ and the noise reduction
controlled by $(I-B)\boldsymbol\omega(t)$. A stronger social influence,
i.e.,~high $\beta_i$, leads to less noise. The contraction effect depends on
the largest singular value of the matrix $(I-B)G'(t)+BS$. 
Given $G'(t)$ and $S$, there always exists a social influence
profile $B$ such that the largest singular value is minimized.  

The overall 
problem is a 
non-convex optimization problem and difficult to solve analytically. We
further simplify the problem by assuming 
$\tilde{g}'_i(t)\equiv\tilde{g}$
is uniform
across all participants,
and the learning function is approximately linear. Similarly, we assume $\beta_i \equiv \beta$ is
uniform across participants. 
The
original dynamics is reduced into the following linear stochastic
dynamics: 
\begin{equation}
  \mathbf{x}(t+1)=\big[(1-\beta) \tilde{g} +\beta S\big]\mathbf{x}(t)+(1-\beta)\boldsymbol\omega(t)
.\end{equation}
We can obtain the upper bound of the expected MSE as
\begin{equation}
\mathbb{E}\big[\operatorname{MSE}(t+1)\big]\leq m^2\operatorname{MSE}(t)+(1-\beta)^2\sigma_\omega^2,\label{eq:mse}
\end{equation}
where $m=(1-\beta) \tilde{g} +\beta$ (see Appendix~\ref{app_math} for detailed derivation). This is a
conservative (or worst case) estimation of the expected MSE evolution. The
practice of computing the optimal control against worst case scenario is
known as robust control. 
In this setting, one computes the control by 
{\em minimizing} the {\em maximum} cumulative expected MSE. 
We denote the robust solution 
as $\beta^\ast_\text{R}$:
\begin{equation}
\beta^\ast_\text{R}=\min_\beta\max_{\boldsymbol\omega(t)}V\big(\beta,\boldsymbol\omega(t);\mathbf{x}(0),T\big).
\end{equation}
The worst case $V(T)$ follows
\begin{equation}
\begin{aligned}
\frac{V(T)}{\operatorname{MSE}(0)}&\leq \left[\frac{1-m^{2T}}{1-m^2}+\right.\\
&\left.\left(T - \frac{1-m^{2T}}{1-m^2}\right)\frac{(1-\beta)^2}{1-m^2}\frac{\sigma_\omega^2}{\operatorname{MSE}(0)}\right]\label{eq:controlRobust}.
\end{aligned}
\end{equation}
In Fig.~\ref{fig:PhaseBasic} we plot the value of $\beta^\ast_\text{R}$ as a function of the
noise-to-initial-MSE ratio $\sigma_\omega^2/\operatorname{MSE}(0)$ and the characteristic learning gain $\tilde{g}$ (given $T=30$ and $n=40$). 
A
moderate 
social influence
  is optimal
when systems are uncertain and one needs the system to equilibrate
quickly.  

\begin{figure}[h!]
	\centering
	\includegraphics[width=\linewidth]{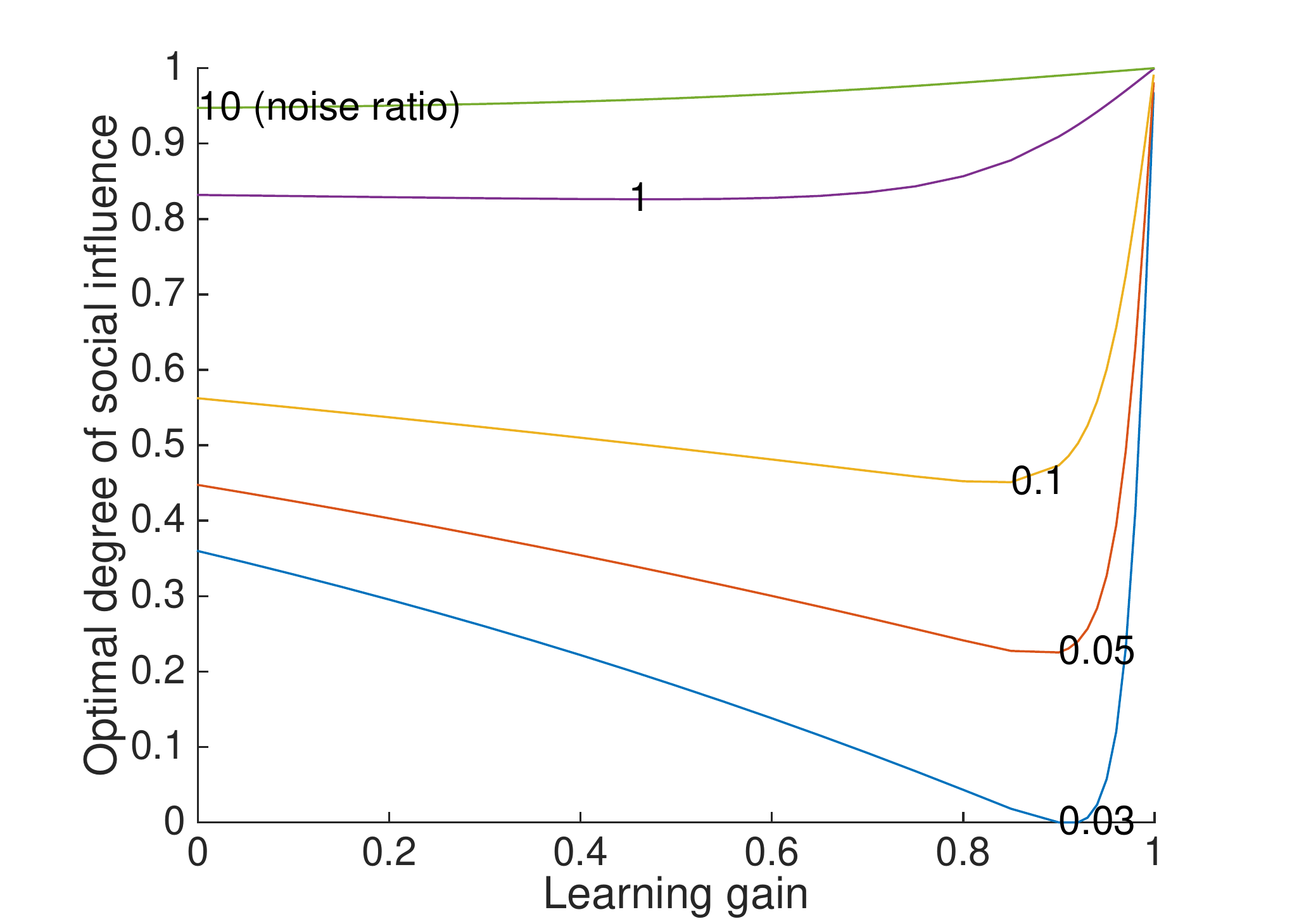}
	\caption{Optimal degree of social influence from robust control by minimizing the RHS of \eqref{eq:controlRobust}. The general trend is that a moderately strong social influence is desirable if the system is uncertain (high noise-to-initial-MSE ratio) or the learning gain is low (fast open loop convergence). An interesting observation is that as the learning gain crosses a certain threshold (e.g.,~0.9), the optimal degree of social influence rapidly increases as the learning gain increases. For a high learning gain, the contraction becomes insensitive to the change in $\beta$ while the noise reduction still does.} 
	\label{fig:PhaseBasic}
\end{figure}


\section{Results}

\subsection{Wisdom of Crowds Effect}
Let's begin with the analysis of the wisdom of crowds effect.
We plot the time series 
of each individual player's decision error ($x_i$) as well as that of the
wisdom of crowds ($u$) in Fig.~\ref{fig:BasicData} (similar to the state
tax and expenditure time series in Fig.~\ref{fig:TaxExp}). 
The performance of the wisdom of crowds
is clearly superior: $u$ steadily and quickly reaches 
the 
solution
within the first minute
while individual players lag behind.

Fig.~\ref{fig:BasicData} 
also confirms the behavior observed in the
literature: The wisdom of crowds significantly outperforms the individual
estimates, but such effect is weakened 
by social influence. The average in the \sr{} setting (red,
right) slightly lags behind that in the open loop (blue, left).  
\begin{figure*}
	\centering
	\includegraphics[width=0.45\linewidth]{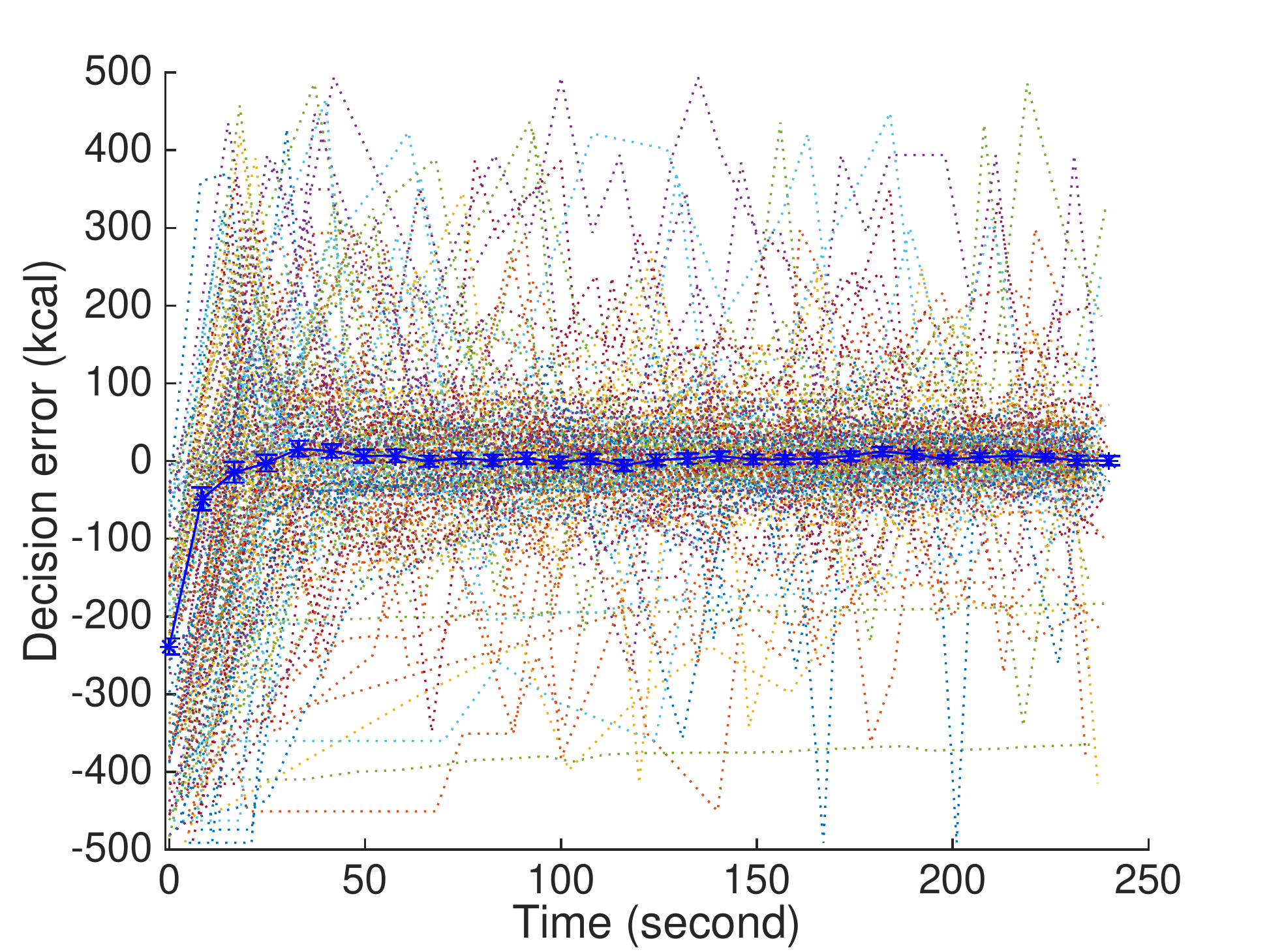}~
	\includegraphics[width=0.45\linewidth]{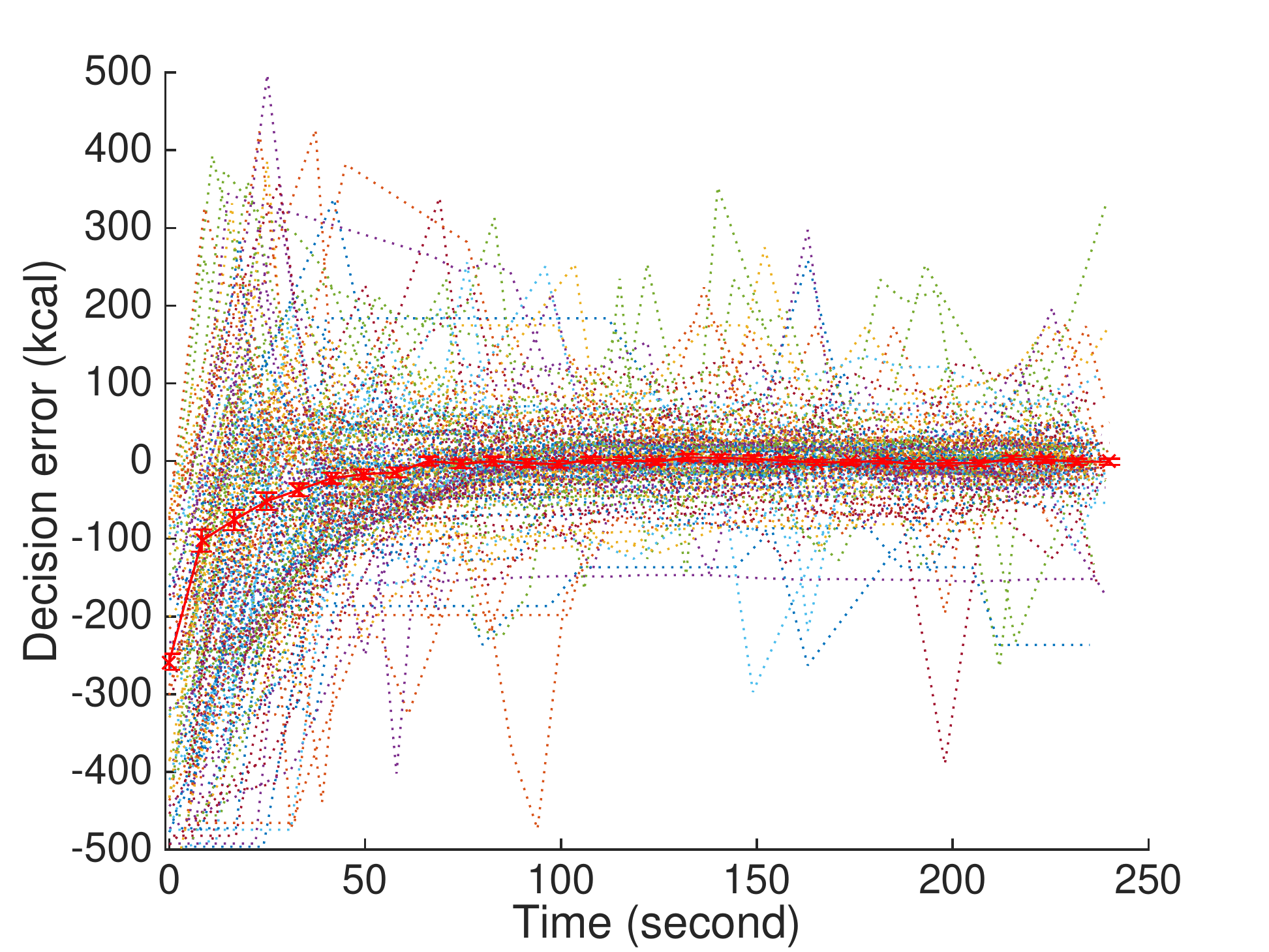}
	\caption{Learning process of each individual player (time series
          of $x_i$) and the wisdom of crowds (time series of $u$)  (left:
          open loop setting, right: \sr{} setting). Each colored
          dashed line represents an individual participant's time series of
          decision error. The solid line is the arithmetic average of
          individual decision errors. Error bars reflect the standard
          errors of the mean.} 
	\label{fig:BasicData}	
      \end{figure*}

\subsection{Improvement from \SR{}}
Now, let's analyze 
how \sr{} improves the crowd's learning performance.
By visually inspecting Fig.~\ref{fig:BasicData}, we observe the narrowing
of individual error distribution in the \sr{} setting: There are
fewer extreme errors than those in the open loop setting; most guesses are
confined within $\pm100$ kcal around optimum. In contrast, there are a
significant number of players making completely off guesses ($\pm500$
kcal) in the open loop (even towards the end of sessions). 

In Fig.~\ref{fig:BasicSysid} we plot the MSE time series to quantitatively
assess the crowd's  
performance. The total MSE 
is approximately $30$\% lower in the \sr{} setting than in the open loop setting.
Unlike the deterioration in the performance of wisdom of crowds,
here social influence 
improves convergence and reduces the effect of
noise. 
{\em The critical feature of \sr{} is that the players
can ignore the feedback}. Since self-interested individuals reject
feedbacks that appear unhelpful, the self-filtered social feedback
significantly improves performance.

The observed improvement from \sr{} indicates that, without
external interference, partially following the average opinion helped the
players solve the ``Fitness Game'' problems. In the next section, we will
characterize the system and estimate how much social influence was present
in the experiment, and the optimal degree of social influence that would
have optimized the crowd's performance. 

\begin{figure}[h!]
	\centering
	\includegraphics[width=\linewidth]{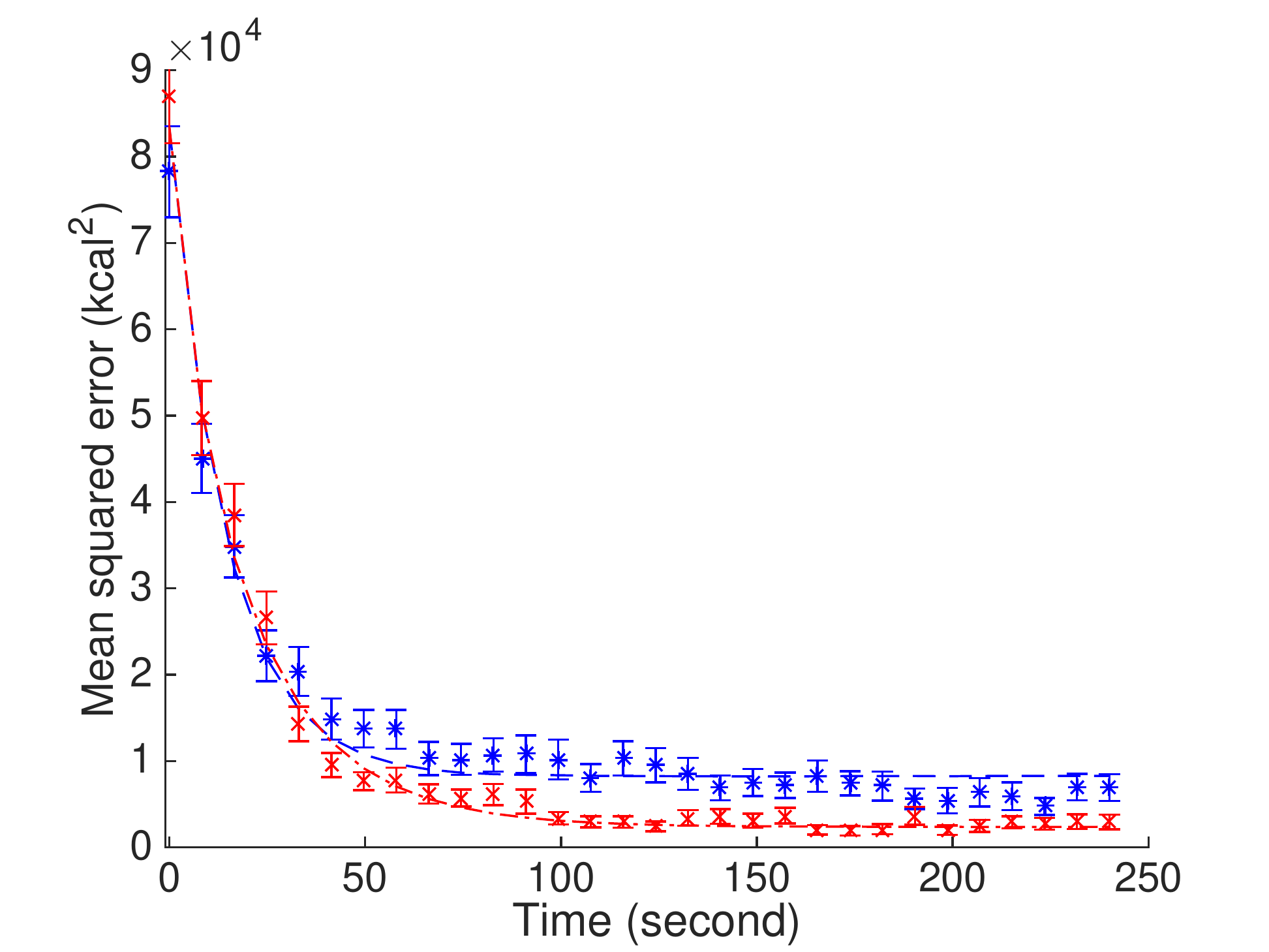}
	\caption{MSE progression. Blue (or red) dots are the MSE values
          sampled at different points in time ($T=30$) in the open loop
          (or \sr{}) setting. The dashed lines are simulation results
          based on models from system identification.} 
	\label{fig:BasicSysid}	
\end{figure}

\subsection{System Identification}
We assumed that 
$g_i(x)\equiv g(x)=\tilde{g}x$
and the degree of social influence
$\beta_i \equiv \beta$. 
The estimate $\hat{g}(x) = 0.75x$ and
$\hat{\sigma}_\omega=60$ ($r^2=0.97$) was computed using the open loop
results. 
From \eqref{eq:mse}, we first estimated $\tilde{g}$ and $\sigma_\omega$ by regressing $\operatorname{MSE}(t+1)$ against $\operatorname{MSE}(t)$. Using these estimates as an initial guess, we then ran a Monte Carlo simulation with 5000 samples and computed the average MSE time series. By minimizing the mean squared difference between that with the open loop MSE time series, we obtained the $\tilde{g}$ and $\sigma_\omega$ estimates.
The corresponding MSE evolution is plotted in Fig.~\ref{fig:BasicSysid}.
  
The estimate $\hat{\beta} = 32\%$ ($r^2=0.99$) for the degree of social influence 
was computed using the results where the players received the population
feedback. 
The corresponding MSE
evolution is plotted in 
Fig.~\ref{fig:BasicSysid}. 
Following the
studies \cite{soll2009strategies,moussaid2013social} that have established that people rely more on
themselves when the opinions of others are very
dissimilar, 
we computed an ``opinion distance'' function $\beta(d)$, where 
$d=|g(x)-u|$ is the distance of an individual decision from the population feedback.
We found it to be 
$\hat{\beta}(d)=\exp(-0.011d)$ ($r^2=0.98$). 


\subsection{Optimal Degree of Social Influence}
\begin{figure*}
	\centering
	\includegraphics[width=0.45\linewidth]{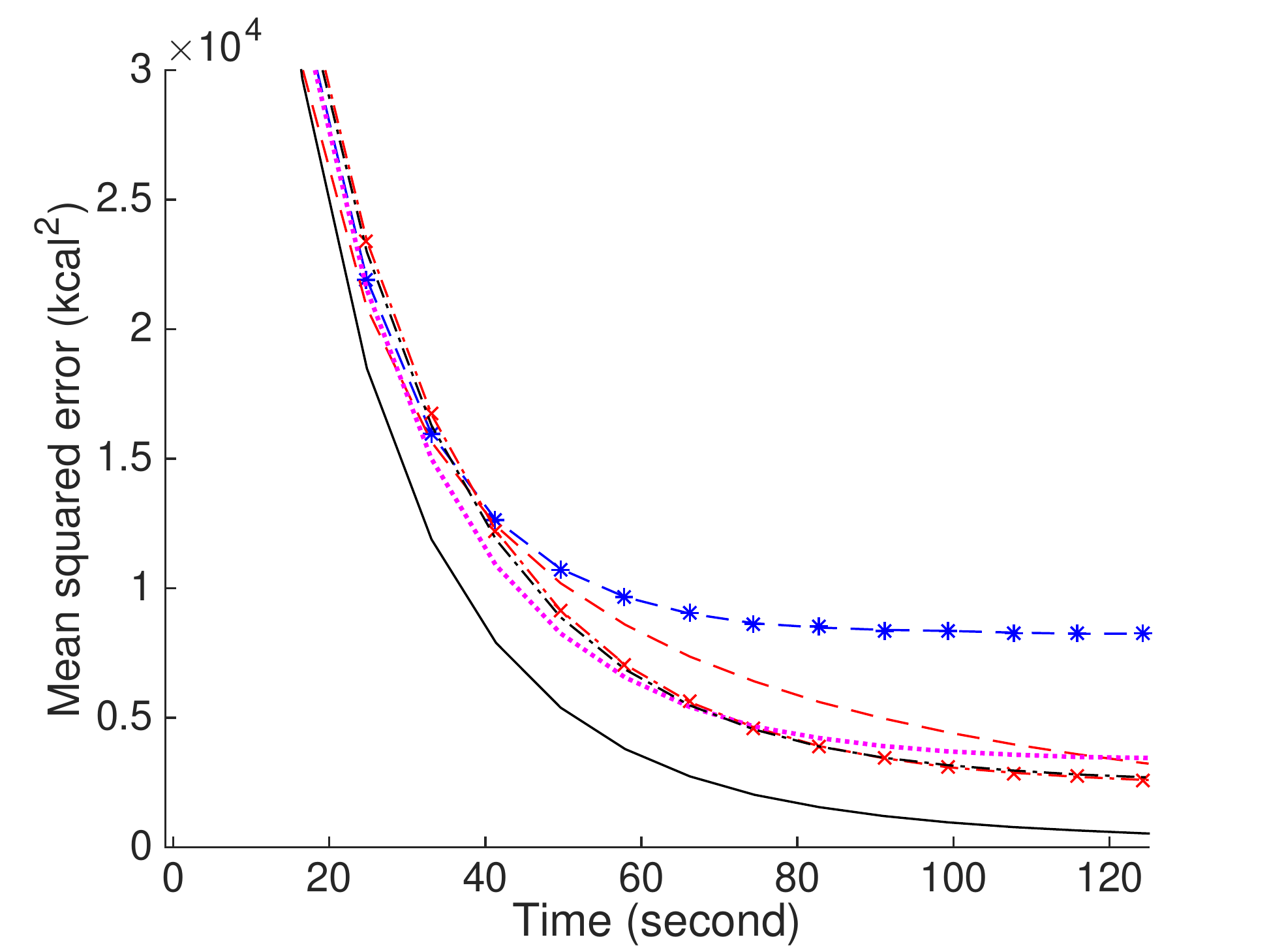}
	\includegraphics[width=0.45\linewidth]{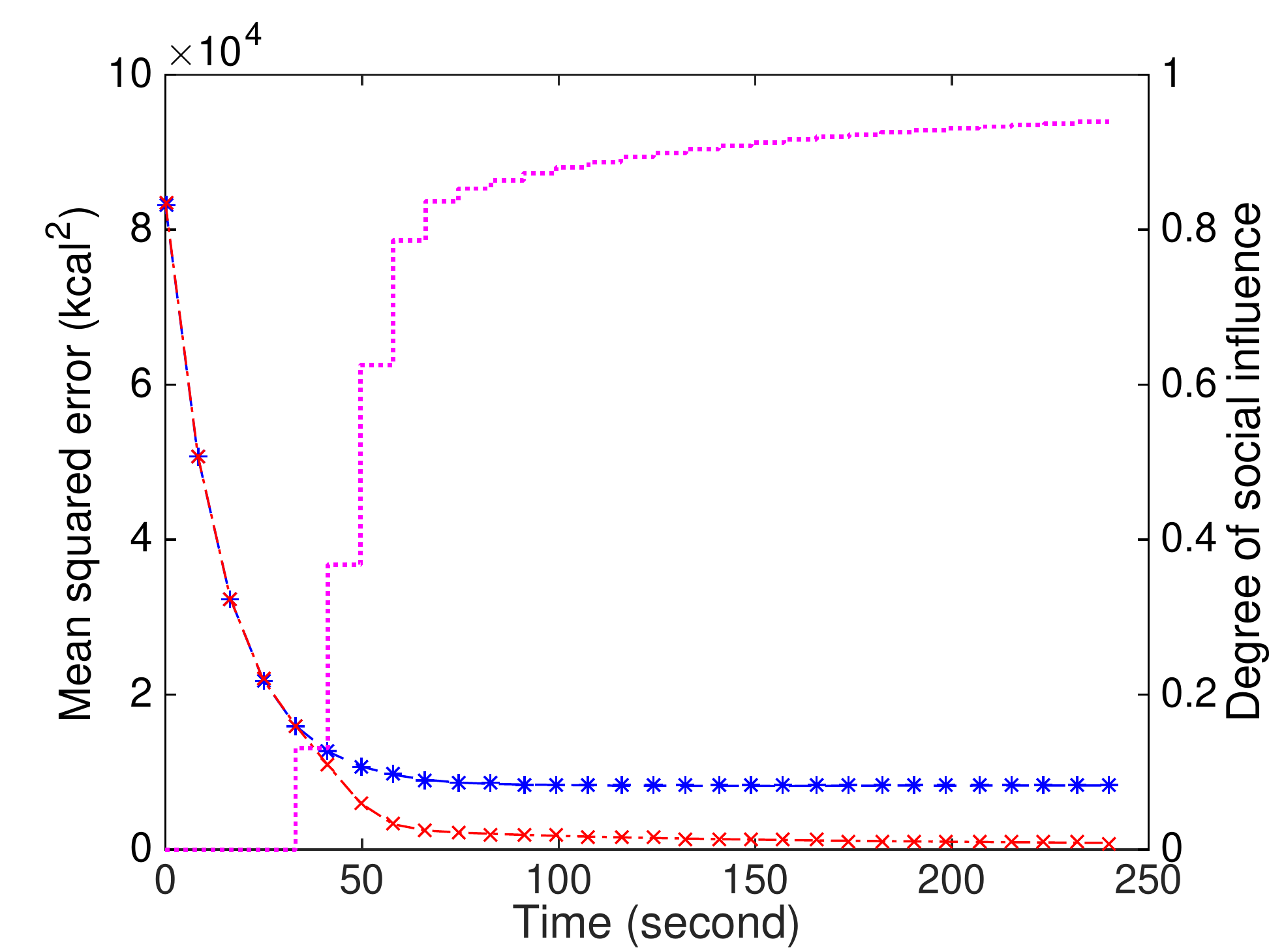}
	\caption{Monte Carlo simulation of the expected MSE time series.
	Blue (or red) dashed line with $\ast$ (or $\times$) markers is the simulation of the open loop (or \sr{}) MSE. Left: Red dashed line is the simulation of the \sr{} MSE with estimate social influence profile $\hat{\beta}(d)$; magenta dotted line is the simulation of MSE with optimal degree of social influence $\beta^\ast_\text{R}$ through robust control; black dash-dot line is with true optimal  degree of social influence $\beta^\ast_\text{MC}$; black solid line is with true optimal social influence profile $\beta^\ast_\text{MC}(d)$. Right: Red dashed line with markers is with dynamic social influence $\beta^\ast_\text{R}(t)$; magenta dotted line is the dynamic social influence time series.}
	\label{fig:BasicControl}	
\end{figure*}

Given the estimates $\hat{g}(x)$ and $\hat{\sigma}_\omega$, 
one can compute the optimal degree of social influence $\beta^\ast$ that, hypothetically,
would optimize the \sr{} performance. The results are summarized
in  Table~\ref{tab:optimal}, and the associated 
MSE time series 
are displayed 
in Fig.~\ref{fig:BasicControl}. 
We first consider the case where the degree of social influence $\beta$ is fixed.
The empirical estimate $\hat{\beta}$  of social influence computed from  experiment data is 
listed 
as a reference.
The \emph{robust}  social influence 
${\beta}^\ast_\text{R}$ 
was calculated by minimizing the RHS in \eqref{eq:controlRobust},
i.e.,~optimizing the worst case cumulative expected
MSE. 
The \emph{Monte Carlo} (MC) estimate ${\beta}^\ast_{\text{MC}}$
was 
calculated 
by minimizing the total MSE in 
\eqref{eq:control1} with the expectation approximated by a Monte Carlo estimate.  
We regard $\beta^\ast_{\rm MC}$ 
as the true optimal degree of social influence.
In Table~\ref{tab:optimal}, the column labeled $\Delta$MSE lists the
decrease of the cumulative expected MSE from the open loop to  the \sr{} setting. 
The performances of the empirical estimate $\hat{\beta}$, the robust estimate
$\beta^\ast_{\text{R}}$, and the optimal value $\beta^\ast_{\text{MC}}$ are
quite close.  
It is comforting to know that 
the social influence present in the experiment was close to the optimum. 

We expect the degree of social influence, a function
of the opinion distance or a function of time, to likely improve
convergence. The $\hat{\beta}(d)$ profile estimated from experimental data
results in $\Delta \text{MSE} = 30\%$, which is not distinguishable from
the performance of a constant $\beta$. However, the optimal $\beta$
profile ${\beta}^\ast_{\text{MC}}(d)$ with $\Delta \text{MSE} = 47\%$ is
significantly superior. The performance of the optimal dynamic robust
social influence $\beta^\ast_{\text{R}}(t)$ is also listed in
Table~\ref{tab:optimal}. 
Since we do not have evidence to suggest the subjects used a dynamic
value for $\beta$, and the performance of $\hat{\beta}(d)$ is close to
$\hat{\beta}$, we assume that the subjects used the constant $\hat{\beta}$ for
the rest of our results. 

\begin{table}[]
\centering
\caption{Optimal degree of social influence}
\label{tab:optimal}
\begin{adjustbox}{max width = \linewidth}
\begin{tabular}{@{}llc@{}}
\hline
Type                 	& Value     								& $\Delta$MSE\\
\hline
 $\beta$ (observed)		& $\hat{\beta}=32\%$ 								& 29\%\\
 $\beta$ (robust)     	& $\beta^\ast_\text{R}=23\%$    			& 27\%\\
 $\beta$ (MC, true optimum) 		& $\beta^\ast_\text{MC}=30\%$             & 29\%\\
$\beta$ profile (observed)		& $\hat{\beta}(d)=\exp(-0.011d)$					& 30\%\\
$\beta$ profile (MC, true optimum)  		& $\beta^\ast_\text{MC}(d)=\exp(-0.026d)$ & 47\%\\ 
Dynamic $\beta$ (robust) 		& $\beta^\ast_\text{R}(t)$ 					& 39\%\\
\hline
\end{tabular}
\end{adjustbox}
\end{table}

\subsection{U.S. State Tax and Expenditure Case Study}
Next, we apply 
this  
control-theoretic analysis to the state tax and expenditure case study. 
The results are displayed in Table~\ref{tab:all_results}. 
The learning gains of the states are all very close to $1$, i.e.,~in a
noiseless setting, the convergence is very slow.  
A possible explanation is that
drastic change of tax and expenditure strategies is either prohibited or
discouranged. 
A larger noise (see e.g.,~T09 and E065) or a smaller learning gain (see
e.g.,~T20 and E65) calls for a larger optimal degree of social influence, which is
consistent with the results presented in  Fig.~\ref{fig:PhaseBasic}. The
improvement from \sr{} ranges from $14$\% to $73$\%. As
mentioned earlier, even a small improvement could make a significant
difference in the nation's overall welfare.  

\begin{table*}[]
\centering
\caption{All Results}
\label{tab:all_results}
\begin{adjustbox}{max width = \linewidth}
\begin{tabular}{@{}lccccccccc@{}}
\hline
Description                  & Duration  & Crowd size ($n$) & Horizon ($T$) & Learning gain ($\hat{\tilde{g}}$) & Noise ($\hat{\sigma}_\omega$) & $r^2$ & Noise ratio & Optimal $\beta$ & $\Delta$MSE \\ \hline
The Fitness Game (Set B)  & 0-240s    & 39               & 30            & 0.75                         & 60                      & 0.97  & 5\%         & 30\%            & 29\%        \\
The Fitness Game (Set N)  & 0-240s    & 41               & 30            & 0.7                          & 57                      & 0.98  & 4\%         & 32\%            & 25\%        \\
The Fitness Game (Set S)  & 0-240s    & 9                & 30            & 0.65                         & 51                      & 0.98  & 3\%         & 30\%            & 17\%        \\
Total Gen Sales Tax (T09)    & 1946-2014 & 50               & 69            & 0.96                         & 4                       & 0.89  & 3\%         & 35\%            & 73\%        \\
Total License Taxes (C118)        & 1946-2014 & 50               & 69            & 0.97                         & 0.82                    & 0.89  & 0.4\%       & 14\%            & 34\%        \\
Alcoholic Beverage Lic (T20) & 1946-2014 & 50               & 69            & 0.93                         & 0.04                    & 0.99  & 0.09\%      & 20\%            & 34\%        \\
Individual Income Tax (T40)  & 1946-2014 & 50               & 69            & 0.98                         & 2.9                     & 0.86  & 1\%         & 14\%            & 32\%        \\
Educ-NEC-Dir Expend (E037)        & 1977-2013 & 51               & 37            & 0.96                         & 0.097                   & 0.85  & 1\%         & 28\%            & 54\%        \\
Emp Sec Adm-Direct Exp (E040)      & 1977-2013 & 51               & 37            & 0.93                         & 0.037                   & 0.99  & 0.6\%       & 11\%            & 14\%        \\
Total Highways-Dir Exp (E065)      & 1977-2013 & 51               & 37            & 0.93                         & 0.76                    & 0.89  & 3\%         & 31\%            & 53\%        \\
Liquor Stores-Tot Exp  (E107)      & 1977-2013 & 51               & 37            & 0.95                         & 0.17                    & 0.95  & 1\%         & 42\%            & 67\%        \\ \hline
\end{tabular}
\end{adjustbox}
\end{table*}

\section{Discussion}
There is a fundamental difference between 
{\it vox populi}
and the 
\sr{} mechanism 
proposed 
in this paper. 
Even though both come under the umbrella of ``collective intelligence,''
the {\it vox populi} aggregates the wisdom of \emph{experts} while the latter harnesses the wisdom of \emph{learners}. Experts 
base their opinions on prior knowledge. Such knowledge comes from experience and beliefs, which are unlikely to change.
Independency and diversity of opinions prevent the ``groupthink'' behavior~---~undesirable convergence of individual estimates~\cite{sunstein2014wiser}. 
In this setting, social influence, which violates independency, 
reduces the accuracy of the wisdom of crowds.

Learners, on the other hand, {\em revise} their decisions by interacting with the problem as well as other learners. 
Consider, for example, flocking birds. 
The birds have to adapt to changing weather; they
gather local information, follow their closest neighbors, and revise
directions constantly~\cite{reynolds1987}. In this collective learning
environment, individuals, like the flocking birds, are \emph{both}
respondents who generate new information, and surveyors who poll their
social networks to improve decisions. 

It appears that a social influence degree of $30$\% 
is robust across many different scenarios.
In Table~\ref{tab:all_results}, the optimal degree of social influence ranges from 30\% to 32\% for the ``Fitness Game'' experiment. 
Prior literature~\cite{soll2009strategies,harvey1997taking,lim1995judgemental,yaniv2004receiving,yaniv2000advice} also reports 30\% to be the commonly observed degree of social influence on average.
Whether this value
 is a mere coincidence requires further investigation. 

The self-interested filtering of the feedback 
is key to ensuring the accuracy and efficiency of the \sr{} mechanism. Individuals 
will 
reject the feedbacks that appear useless.
The experimentally observed magnitude of \sr{} 
is close to the theoretically predicted value for the 
optimal degree of social influence.  
This discovery suggests the promise of \sr{} for challenging real-world problems that require collective learning and action.

\appendices

\section{The ``Fitness Game''}\label{app_game}
All experiments have been approved by the IRB of
Columbia University (Protocol Number: IRB-AAAQ2603). We developed the ``Fitness Game'' using Google Apps Script and conducted
the experiments on Amazon Mechanical Turk (AMT). All the data were stored in
Google Sheets.  Once the players accepted the task on
AMT, they  were first asked to 
carefully read the game
instructions (see Fig.~\ref{fig:ScreenIntro}). The total  task duration
was ten minutes. The open loop (game level 
1) and \sr{} (game level 2) sessions lasted precisely four
minutes each. Players who wished to practice could enter the practice mode
(game level 0) any time before open loop session began. After completing
both open loop and \sr{} sessions, the players received a
message about compensation information. 

The interactive app (see Fig.~\ref{fig:ScreenInterface}) consists of the
following components: The upper left panel shows the number of attempted
guesses, the most recent guess, the fitness level, and the latest
score. The panel changes from red to green whenever the player earns one
point. In the \sr{} session, an additional message recommends
the current {\it vox populi} population feedback (see
Fig.~\ref{fig:ScreenInterface}, right). The 
upper right panel records latest game scores. The lower left scatter chart
plots the ten most recent entries (fitness versus diet). The lower right
line chart plots the fitness history of the ten most recent entries. 
\begin{figure*}
	\centering
	\includegraphics[width=0.8\linewidth]{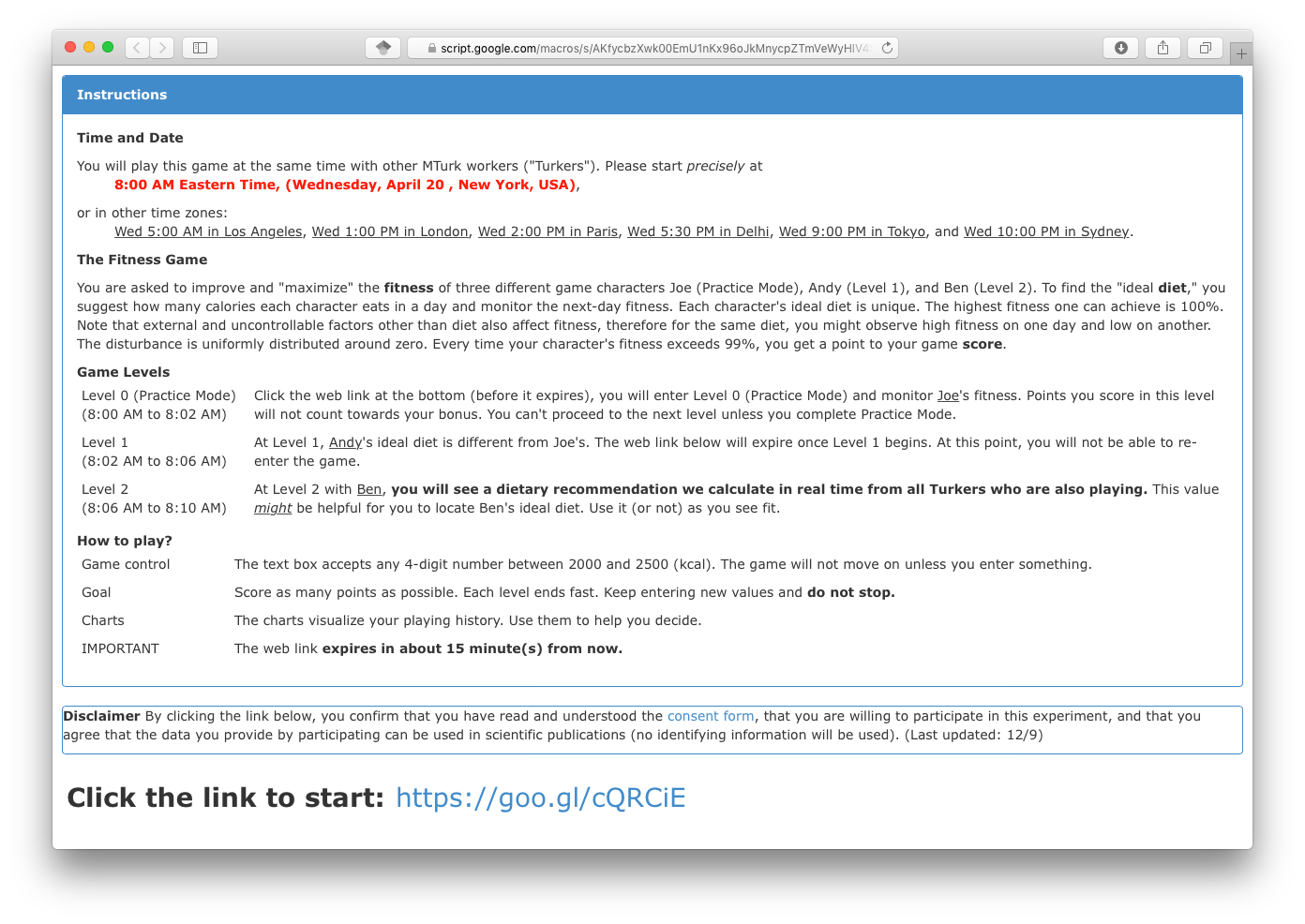}
	\caption{{Instructions.}}
	\label{fig:ScreenIntro}
\end{figure*}

\begin{figure*}
    \centering
 		\includegraphics[width=0.45\linewidth]{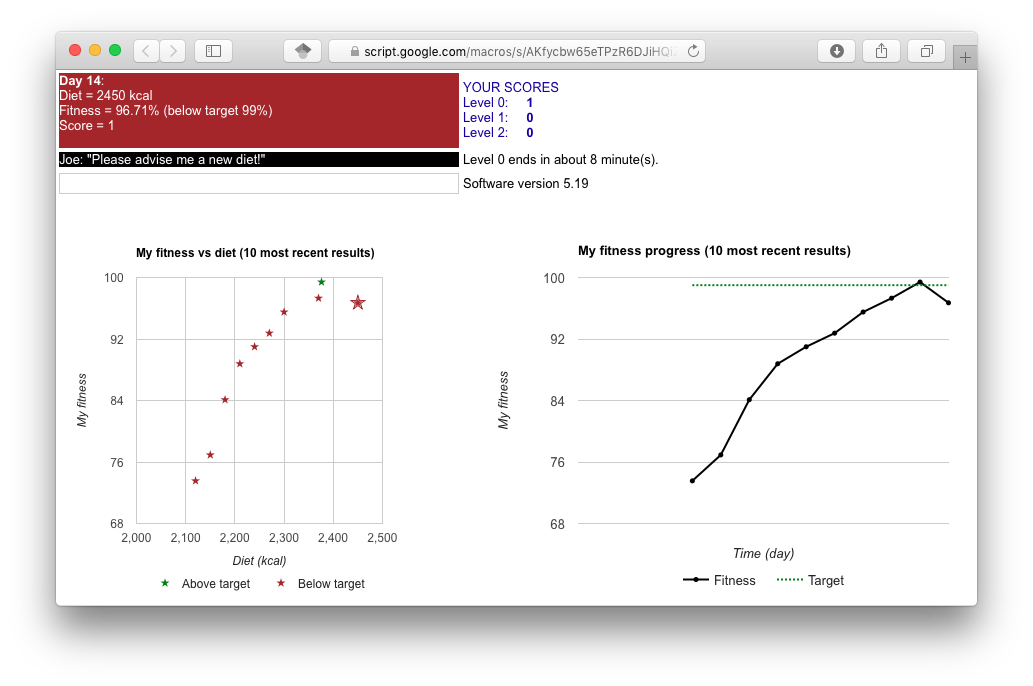}
    	\includegraphics[width=0.45\linewidth]{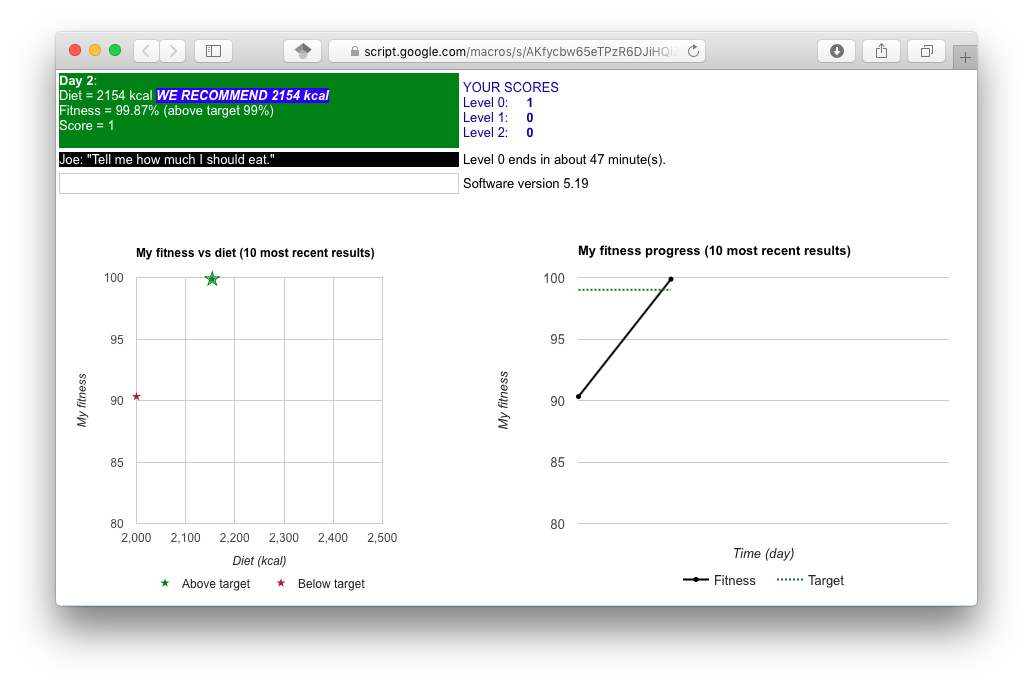}
    \caption{{ Game interface.}} 
    \label{fig:ScreenInterface}
\end{figure*}

The virtual character's random fitness level $f(x)$ as a function of the
input $x=z-\theta^\ast$ was given by 
\begin{equation*}
f(x)=f_0-\Big(\frac{x}{\kappa}\Big)^2+\tilde{\omega},
\label{eq:fitness}
\end{equation*}
where $f_0=98\%$ is the maximum achievable fitness, $\kappa = 500~\text{kcal}$ is the scale of the fitness function, and $\tilde{\omega}$ is a sample from a random variable uniformly distributed over $[-2\%,2\%]$.
 The player was awarded one score point whenever the guess led to a fitness level of 99\% or higher.

\section{Mathematical Model and Proofs}\label{app_math}
We first begin with the noiseless dynamics and then extend the model to
include noise. The noiseless dynamics for the $n$-player ``Fitness Game''
is as follows:
\begin{equation*}
x_i(t+1)=(1-\beta_i)g_i\big(x_i(t)\big)+\beta_iu(t)\label{eq:noise-less-dynamics}
,\end{equation*}
where $x_i(t)$ is the $i$-th player's state, i.e., deviation from
optimum $\theta^\ast$ at time $t$,  and restricted to belong to a bounded set
$\mathbb{X}\subseteq\mathbb{R}$, the learning function $g_i(\cdot)$ denotes the player's own
state update process, $\beta_i \in [0,1]$ (degree of social influence) 
is the weight player~$i$ puts on the 
soft feedback
$u(t)= \frac{1}{n}\sum_ix_i(t)$. 
Note that in the paper, we refer to $\beta$ as percentage. The individual
learning functions $\{g_i(x): 1 \leq i \leq n\}$ are assumed to
satisfy the following regularity condition:\\
\begin{assumption}
  \label{ass:g-assumption}
  For all $i \in \{1, \ldots, n\}$, the function $g_i$ is differentiable,
  $x^\ast=0$ is the unique attracting fixed point of  
  $g_i$, and furthermore,  
  $g_i$ is a contraction,~
  i.e., 
  $|g_i'(x)|<1$
  for all $1\leq i\leq n$ and $x\in\mathbb{X}$. 
\end{assumption}
This assumption is
motivated by the fact that all players converged to the optimal point in
the open loop setting independent of the starting guess.

Let $\mathbf{x}\equiv[x_1,\dots,x_n]^\top\in\mathbb{X}^n$ denote the state
vector for the $n$ agents. 
The \sr{} map for the vector $\mathbf{x}$ is given by $\bx(t+1)
= \mathbf{h}\big(\bx(t)\big)$ where the map
\begin{eqnarray*}
\mathbf{h}(\bx) &=& 
\begin{bmatrix}
(1-\beta_1)g_1(x_1)&\dots&0\\
\vdots&\ddots&\vdots\\
0&\dots&(1-\beta_n)g_n(x_n)
\end{bmatrix}
+\\ 
& & \frac{1}{n}\begin{bmatrix}
  \beta_1 \\
  \vdots\\
  \beta_n
\end{bmatrix}\mathbf{1}^\top \bx
.\end{eqnarray*}
We first show that the state vector $\bx(t)$ converges to $\bx^\ast =
\mathbf{0}$ if the functions $\{g_i(x): 1\leq i \leq n\}$ satisfy
Assumption~\ref{ass:g-assumption} and $0 \leq \max_i \beta_i < 1$.\\

\begin{theorem}
The spectral radius $\rho\big(J(\mathbf{x})\big)$ of the Jacobian matrix of the
\sr{} map $\mathbf{h}(\mathbf{x})$ 
satisfies 
$\rho\big(J(\bx)\big) \leq m=\max_{1\leq i\leq
  n,x\in\mathbb{X}}\Big\{(1-\beta_i)|g_i'(x)|+\beta_i\Big\} < 1$.  
\end{theorem}
\begin{proof}
 The Jacobian of $\mathbf{h}(\bx)$ is 
\begin{eqnarray*}
J(\mathbf{x}) &=&
\operatorname{diag}\big((1-\beta_1)g_1'(x_1),\dots,(1-\beta_n)g_n'(x_n)\big) 
+ \\
&&\frac{1}{n} \begin{bmatrix}
  \beta_1 \\ 
  \vdots\\
  \beta_n
\end{bmatrix} \mathbf{1}^\top
.\end{eqnarray*}
The induced $\infty$-norm $\Vert J(\mathbf{x})\Vert_\infty$ of the
Jacobian $J$ 
satisfies
\begin{eqnarray*}
\Vert
J(\mathbf{x})\Vert_\infty   
&=& \max_{\Vert
\mathbf{v}\Vert_\infty = 1} \Vert
J(\mathbf{x}) \mathbf{v}\Vert_\infty, \\
  &  = &  \max_{\Vert
\mathbf{v}\Vert_{\infty} = 1} \max_{1\leq i\leq
  n} |J_i(\bx) \mathbf{v}|,\\
  & = & \max_{\Vert
\mathbf{v}\Vert_\infty = 1} \max_{1\leq i\leq
  n}  \Big[(1-\beta_i)|g_i'(x_i)|v_i+\\
  &&\frac{1}{n}\beta_i
(\mathbf{1}^\top \mathbf{v}) \Big],\\ 
& \leq & m(\bx),
\end{eqnarray*}
where $J_i(\bx)$ denotes the $i$-th row of the Jacobian $J(\bx)$. 
The result follows from noting that $\rho\big(J(\mathbf{x})\big) \leq \Vert
J(\mathbf{x})\Vert_\infty = m(\bx)$. It is easy to see that $m(\bx) < 1$
whenever $\max_i \beta_i < 1$.
\end{proof}
This result immediately implies that $\bx^\ast = 0$ is an asymptotically
stable fixed point of the map~$\mathbf{h}(\bx)$.\\


\begin{theorem}
The fixed point $\mathbf{x}^\ast  = \mathbf{0}$ of the map $\mathbf{h}$ is robust
when subjected to bounded disturbances. 
\end{theorem}
\begin{proof}
Let $V(\mathbf{x})=\Vert\mathbf{x}\Vert_{\infty}$. 
Since $\mathbf{h}(\mathbf{0}) = \mathbf{0}$, the mean value theorem implies that 
\[
\mathbf{h}(\bx) = \begin{bmatrix}
  J_1(\delta_1 \bx)\\ 
  \vdots\\
  J_n(\delta_n \bx) 
\end{bmatrix}
\bx,
\]
for some $\delta_i \in [0,1]$, $i = 1, \ldots, n$, and $J_i(\delta_i \bx)$
denotes the $i$-th row of the Jacobian of 
$\mathbf{h}(\delta_i \bx)$.
Thus, 
\begin{align*}
V\big(\mathbf{h}(\bx)\big) &= \Vert \mathbf{h}(\bx) \Vert_{\infty},\\ 
&= 
\max_{1 \leq i \leq n} |J_i(\delta_i \bx) \bx|, \\
&\leq \Big(\max_{1 \leq i \leq n}
\Vert J(\delta_i \bx)\Vert_{\infty} \Big) \Vert \bx \Vert_{\infty},\\ 
&< m \Vert \bx \Vert_{\infty},
\end{align*}
where the first inequality follows from the definition of $\Vert J(\delta_i
\bx)\Vert_{\infty}$. 

Since the  continuous function $V(\mathbf{x})$ is a Lyapunov function for
$\mathbf{h}$,  
the result follows from standard results in 
stability theory~\cite{teel2004discrete}. 
\end{proof}


In the rest of this section, we will assume that $\beta_i$ are identically
equal to $\beta$.\\
\begin{theorem}
  \label{thm:Hx-bnd}
  Suppose $\beta_i$ are all identically equal to $\beta$. Then 
  $\norm{\mathbf{h}(\bx)}_2 \leq m \norm{\bx}_2$, where $m =(1-\beta) \max_{1 \leq i \leq n, x \in \mathbb{X}} |g_i'(x)| +\beta$. 
\end{theorem}
\begin{proof}
Using the mean value theorem, one can write 
\[
\mathbf{h}(\bx) = \Big((1-\beta) \text{diag}\big(g_1'(\delta_1 x_1),\dots,g_n'(\delta_n
x_n)\big)  + \frac{\beta}{n}\ones\ones^\top \Big) \mathbf{x},
\]
where $\delta_i \in [0,1]$ for $i = 1, \ldots, n$.
Let $G' = \text{diag}(g_1'(\delta_1 x_1),\dots,g_n'(\delta_n
x_n))$ and $J = (1-\beta) G' + \frac{\beta}{n} \ones\ones^\top$.
Then
\begin{eqnarray*}
  \norm{G'}_2^2  &=&  \max_{\norm{\bv}_2 = 1} \norm{G'\mathbf{v}}_2^2,\\
  &=&  \max_{\norm{\bv}_2 = 1} \sum_i |g'_i(\delta_i x)|^2
        v_i^2,\\
  &\leq&  \max_{1 \leq i \leq n} |g'_i(\delta_i x_i)|^2 \\
  &\leq& \max_{1 \leq i
           \leq n, x \in \mathbb{X}} |g_i'(x)|^2. 
\end{eqnarray*}
Thus, $\norm{G'}_2 \leq \max_{1 \leq i
           \leq n, x \in \mathbb{X}} |g_i'(x)| $.
Therefore,
\begin{eqnarray*}
  \norm{J}_2^2 & = & \max_{\norm{\bv}_2 = 1} \norm{J \bv}_2^2, \\
  & = &
                                                         \max_{\norm{\bv}_2
                                                         = 1}
                                                         \left\{(1-\beta)^2
                                                         \norm{ G'
                                                         \bv}_2^2 +
                                                         \frac{\beta^2}{n^2}
                                                         (\ones^\top
                                                         \bv)^2
                                                         \norm{\ones}_2^2
                                                         +\right.\\
                                                         & &\left.\frac{2\beta(1-\beta)}{n}
                                                         (\ones^\top \bv)
                                                         (\ones^\top G'\bv)
                                                         \right\},\\
  &  \leq & (1-\beta)^2 \norm{G'}_2^2 + \beta^2 + \\
            &&\frac{2\beta(1-\beta)}{n} \big( \max_{\norm{\bv}_2 = 1}
            |\ones^\top v| \big) \big( \max_{\norm{\bv}_2 =
            1}|\ones^\top G'\bv|\big),\\
  & \leq & (1-\beta)^2 \norm{G'}_2^2  + \beta^2 +\\
           &&\frac{2\beta(1-\beta)}{\sqrt{n}} \norm{\ones}_2 \big( \max_{\norm{\bv}_2 =
            1}\norm{G' \bv}_2\big),\\
  & = & (1-\beta)^2 \norm{G'}_2^2  + \beta^2 +
           2\beta(1-\beta)\norm{G'}_2 = m^2.
\end{eqnarray*}
Since $\mathbf{h}(\bx) = J\bx$, it follows that $\norm{\mathbf{h}(\bx)}_2 = \norm{J \bx}_2
\leq \norm{J}_2 \norm{\bx}_2 \leq m \norm{x}_2$.
\end{proof}

Next, we introduce noise in the game dynamics. Let 
$\{\bfomega(t) \in \mathbb{R}^n: t \geq 0\}$ denote an IID sequence of random
vectors where $\bfomega(t) =[\omega_1(t),\dots,\omega_n(t)]^\top$, and 
each $\omega_i(t)$ is an IID sample 
of a zero mean random variable with variance
$\sigma_{\omega}^2$. 
The noisy
game dynamics is given by 
\begin{equation*}
x_i(t+1)=(1-\beta_i)\Big(g_i\big(x_i(t)\big) + \omega_i(t)\Big) +\beta_i u(t),\label{eq:noisy-dynamics}
\end{equation*}
i.e.,~we replace $g_i\big(x_i(t)\big)$ by the noisy state update $g_i\big(x_i(t)\big) +
\omega_i(t)$. This modification models the fact that the players sample a
noisy version of the fitness function, and use these noisy samples to
generate the update; therefore, we expect the state update to be
noisy. Note that the noise is \emph{not} measurement noise, rather noise in the
function evaluation. 


Define the mean squared error MSE  
\begin{equation*}
\operatorname{MSE}(t)=\frac{1}{n}\sum_ix_i(t)^2=\frac{1}{n}\norm{\bx(t)}_2^2
.\end{equation*}
Then 
\begin{eqnarray}
\lefteqn{\mathbb{E}\big[\operatorname{MSE}(t+1)\mid \bx(t)\big]}\nonumber\\
  &=&\frac{1}{n}\mathbb{E}\Big[ \norm{\bx(t+1)}_2^2 \mid \bx(t) \Big],
    \nonumber \\
  &=&\frac{1}{n}\mathbb{E}\Big[\norm{\mathbf{h}(\bx(t))+(1-\beta)\bfomega(t)}_2^2 \mid
    \bx(t)\Big], \nonumber \\
  &=&\frac{1}{n}\norm{\mathbf{h}(\bx(t))}_2^2+\frac{(1-\beta)^2}{n}\mathbb{E}\big[\Vert
    \bfomega(t)\Vert_2^2\big], \label{eq:bnd-1}\\
  &\leq&
    \frac{m^2}{n}\Vert \bx(t)\Vert_2^2+(1-\beta)^2\sigma_\omega^2,
    \label{eq:bnd-2}\\
  &=&m^2\operatorname{MSE}(t)+(1-\beta)^2\sigma_\omega^2, \label{eq:bnd-3}
\end{eqnarray}
where \eqref{eq:bnd-1} follows from the fact that $\bfomega(t)$ is
independent of $\bx(t)$, and \eqref{eq:bnd-2} follows from the bound in
Theorem~\ref{thm:Hx-bnd}. Iterating the bound \eqref{eq:bnd-3} we get
\begin{equation*}
  \label{eq:bnd-4}
  \mathbb{E}[\text{MSE}(t)] \leq m^{2t} \text{MSE}(0) + \frac{(1-\beta)^2
    (1-m^{2t})}{(1-m^2)} \sigma_{\omega}^2.
\end{equation*}
\section*{Acknowledgment}

This work is supported in part by Center for the Management of Systemic Risk and Columbia University.




\bibliographystyle{IEEEtran}
\bibliography{sr2-references}
\end{document}